\documentclass[a4paper,12pt]{amsart}
\usepackage{amsmath,amsthm,amssymb}
\usepackage{hyperref}

\allowdisplaybreaks[1]

\DeclareMathOperator{\RE}{Re}
\DeclareMathOperator{\IM}{Im}

\newcommand{\ti}{\widetilde}

\newcommand{\ity}{\infty}
\newcommand{\C}{\mathbb{C}}

\newcommand{\N}{\mathbb{N}}

\newcommand{\B}{\Big}

\numberwithin{equation}{section}
\newtheorem{theorem}{Theorem}[section]

\theoremstyle{remark}
\newtheorem{remark}[theorem]{Remark}

\newtheorem{definition}[theorem]{Definition}

\thanks {The research work of the first  author is supported by research fellowship from Council of Scientific and Industrial Research (CSIR), New Delhi and of the second author by research fellowship from UGC, New Delhi.}

\begin{document}

\title[Dynamics of composite entire functions]{Dynamics of composite entire functions}

\author[D. Kumar]{Dinesh Kumar}
\address{Department of Mathematics, University of Delhi,
Delhi--110 007, India}

\email{dinukumar680@gmail.com }

\author[G.Datt]{Gopal Datt}
\address{Department of Mathematics, University of Delhi,
Delhi--110 007, India}

\email{datt.gopal@ymail.com}

\author[S. Kumar]{Sanjay Kumar}

\address{Department of Mathematics, Deen Dayal Upadhyaya College, University of Delhi,
New Delhi--110 015, India }

\email{sanjpant@gmail.com}

\begin{abstract}
It is known that the dynamics of $f$ and $g$ vary to a large extent from that of its composite entire functions. Using Approximation theory of entire functions, we have shown the existence of entire functions $f$ and $g$ having infinite number of domains satisfying various properties and relating it to their composition. We have explored and enlarged all the maximum possible ways of the solution in comparison to the  past result worked out.
\end{abstract}

\keywords{ Fatou set, wandering domain, Carleman set, preperiodic domain}

\subjclass[2010]{37F10, 30D05}

\maketitle

\section{Introduction}\label{sec1}

Let $f$ be a transcendental entire function. For $n\in\N,$ let $f^n$ denote the $n$-th iterate of $f.$ The set $F(f)=\{z\in\C : \{f^n\}_{n\in\N}\,\text{ is normal in some neighborhood of}\, z\}$ is called the Fatou set of $f$ or the set of normality of $f$ and its complement $J(f)$ is the Julia set of $f$. The Fatou set is open and completely invariant: $z\in F(f)$ if and only if $f(z)\in F(f)$ and consequently $J(f)$ is completely invariant. The Julia set of a transcendental entire function is non empty, closed perfect set and unbounded. All these results and more can be found in Bergweiler \cite{berg1}. If $U$ is a component of Fatou set $F(f),$ then $f(U)$ lies in some component $V$ of $F(f)$ and  $V\setminus f(U)$ is a set which contains atmost one point  \cite{berg3}. This result was also independently proved in \cite{MH}. A component $U$ of  Fatou set of $f$ is called a wandering domain if $U_k\cap U_l=\emptyset$ for $k\neq l,$ where $U_k$ denotes the component of $F(f)$ containing $f^k(U),$ otherwise $U$ is called a preperiodic component of $F(f),$ $f^k(U_l)=U_l$ for some $k,l\geq 0.$ If $f^k(U)=U,$ for some $k\in\N,$ then $U$ is called a periodic component of $F(f).$ Sullivan \cite{Sullivan} proved that the Fatou set of any rational function has no wandering domains. It was Baker \cite{baker2} who gave the first example of an entire function having wandering domains. Thereafter  several other examples of wandering domains  have been given by various authors \cite{APS3}. Certain classes of transcendental functions which do not have wandering domains are also known  \cite{baker3, baker7, berg1, berg2, el2, keen, stallard}.

Two functions $f$ and $g$ are called permutable if $f\circ g=g\circ f.$
Fatou  \cite{beardon}, proved that if $f$ and $g$ are two rational functions which are permutable, then $F(f)=F(g)$.
This was the most famous result  that motivated the dynamics of composition of complex functions. Similar results for transcendental entire functions are still not known, though it holds in some very special cases \cite{baker3}.
If $f$ and $g$ are transcendental entire functions, then so is $f\circ g$ and $g\circ f,$ and the dynamics of one composite entire function helps in the study of the dynamics of the other and vice-versa. A complex number $w\in\C$ is a critical value of a transcendental entire function $f$ if there exist some $w_0\in\C$ with $f(w_0)=w$ and $f'(w_0)=0.$ Here $w_0$ is called a critical point of $f.$ The image of a critical point of $f$ is  critical value of $f.$ Also  $\zeta\in\C$ is an asymptotic value of a transcendental entire function $f$ if there exist a curve $\Gamma$ tending to infinity such that $f(z)\to \zeta$ as $z\to\ity$ along $\Gamma.$  In \cite{dinesh2}, the authors considered the relationship between Fatou sets and singular values of transcendental entire functions $f, g$ and $f\circ g$. They gave various conditions under which Fatou sets of $f$ and $f\circ g$ coincide and also considered relation between the singular values of $f, g$ and their compositions. 
A natural extension of the dynamics associated to the iteration of a complex function is the dynamics of composite of two or more such functions and this leads to the realm of semigroups of transcendental entire functions. In this direction, the seminal work was done by Hinkkanen and Martin \cite{martin} related to semigroups of rational functions. In their  paper, they extended the classical theory of the dynamics  associated to the iteration of a rational function of one complex variable to the more general setting of an arbitrary semigroups of rational functions. 
Many of the results were extended to semigroups of transcendental entire functions in   \cite{{cheng}, {dinesh1}, {poon1}, {zhigang}}. 
In \cite{dinesh1}, the authors generalised  the dynamics of a transcendental entire function on its Fatou set to the dynamics of semigroups of transcendental entire functions. They also investigated the dynamics of conjugate semigroups, abelian transcendental semigroups and wandering and Baker domains of transcendental semigroups. Recall that if $g$ and $h$ are transcendental entire functions and  $f$ is a continuous map of the complex plane into itself with $f\circ g=h\circ f,$ then $g$ and $h$ are said to be semiconjugated by $f$ and $f$ is called a semiconjugacy \cite{berg5}. In \cite{dinesh3}, the first author  considered the dynamics of semiconjugated entire functions and  provided several conditions under which the semiconjugacy  carries Fatou set of one entire function into  Fatou set of other entire function appearing in the semiconjugation. Furthermore, it was shown that under certain conditions on the growth  of  entire functions appearing in the semiconjugation, the set of  asymptotic values of the derivative of composition of the entire functions is bounded.
 The following theorem is a well known result  \cite{berg5, berg6, Poon}.
\begin{theorem}\label{sec1,thm1}
Let $f$ and $g$ be two non-linear entire functions. Then $f\circ g$ has  wandering domains if and only if $g\circ f$ has  wandering domains.
\end{theorem}
It is also known that the dynamics of $f\circ g$ and  $g\circ f$ are very similar. Singh ~\cite{APS3} constructed several examples where the dynamics of $f$ and $g$ vary largely from the dynamics of the composite entire functions. He proved the following result:
\begin{theorem}\label{sec1,thm2}
 There exist transcendental entire functions $f$ and $g$ such that
\begin{enumerate}
\item [(i)] there is a domain which lies in the wandering component of $f$ and wandering component of $g$ and lies in the periodic component of $g\circ f;$
\item [(ii)] there is a domain which lies in the wandering component of $f$ and wandering component of $g$ and also lies in the wandering component of $f\circ g$ and the wandering component of $g\circ f;$
\item [(iii)] there is a domain which lies in the  periodic component of $f$ and  periodic component of $g$, but lies in the wandering component of $f\circ g$ and the wandering component of $g\circ f;$
\item [(iv)]  there is a domain which lies in the  periodic component of $f$ and  periodic component of $g$ and also in the periodic component of $g\circ f$ but lies in the wandering component of $f\circ g.$
\end{enumerate}
\end{theorem}
In the construction of the proof in \cite{APS3}, the author has exhibited entire functions $f$ and $g$ with one domain $G_1$ satisfying the conditions of Theorem \ref{sec1,thm2}. In this connection, one would also be interested in knowing whether it is possible to have entire functions $f$ and $g$ having more than one domain satisfying the conditions of Theorem \ref{sec1,thm2}. This is indeed possible. We have shown the existence of  entire functions having infinitely many domains satisfying the conditions of Theorem \ref{sec1,thm2}. Moreover, we have constructed several other examples where the dynamical behavior of $f$ and $g$ vary greatly from the dynamical behavior of $f\circ g$ and $g\circ f.$  We have explored and enlarged all the maximum possible ways of the solution in comparison to the  past result worked out.
We will require the following definition of Carleman set and a result from Approximation theory of entire functions to prove the theorems in the next section, see \cite{DG}:
\begin{definition}\label{sec1,defn1}
Let $S$ be a closed subset of $\C$ and 
\[C(S)=\{h:S\to\C\, |\, h \,\text{is continuous on} \,S \,\text{and analytic in the interior}\,S^\circ \,\text{of}\, S\}.
\]
Then $S$ is called a Carleman set (for $\C$) if for any $f\in C(S)$ and any positive continuous function $\epsilon$ on $S$, there exist an entire function $g$ such that $|f(z)-g(z)|<\epsilon(z),$ for all $z\in S.$
\end{definition}
\begin{theorem}\cite{DG}\label{sec1,thm0}
 Let $S$ be a closed proper subset of $\C.$ Then $S$ is a Carleman set in $\C$ if and only if $S$ satisfies:
\begin{enumerate}
\item [(i)] $\ti\C\setminus S$ is connected;
\item [(ii)] $\ti\C\setminus S$ is locally connected at $\infty$;
\item [(iii)] for every compact subset $K$ of $\C$ there exist a neighborhood $V$ of $\infty$ in $\ti\C$ such that no component of $S^\circ$ intersects both $K$ and $V.$
\end{enumerate}
\end{theorem}

 In our investigation, $f$ and $g$ are always  assumed to be  transcendental entire functions.
\section{Theorems and their proofs}
\begin{theorem}\label{sec2,thm1}
There exist transcendental entire functions $f$ and $g$ having infinite number of domains which lies in the wandering component of $f$ and wandering component of $g$ and lies in the periodic component of $g\circ f.$
\end{theorem}

\begin{proof}
We follow the construction of Carleman set as in \cite{APS3}.
Let
\[
S=G_0\cup\B{\{{\bigcup}_{k=1}^{\ity}(G_k\cup B_k\cup L_k\cup M_k)\B\}},
\]
where $G_0=\{z:|z-2|\leq 1\},\\
G_k=\{z:|z-(4k+2)|\leq 1\}\cup\{z:\RE z=4k+2\;\text{and}\;\IM z\geq 1\}\cup\{z:\RE z=4k+2\;\text{and}\;\IM z\leq-1\}, \;k=1,2,\ldots,\\
M_k=\{z:\RE z=-4k\},\;k=1,2,\ldots,\\
L_k=\{z:\RE z=4k\},\;k=1,2,\ldots, \;\mbox{and}\\
B_k=\{z:|z+(4k+2)|\leq 1\}\cup\{z:\RE z=-(4k+2)\;\text{and}\;\IM z\geq 1\}\cup\{z:\RE z=-(4k+2)\;\text{and}\;\IM z\leq-1\}, \;k=1,2,\ldots.$
Then using Theorem \ref{sec1,thm0}, we get $S$ is a Carleman set. 
It is known that the set of all natural numbers $\N$ can be expressed in an infinite array of numbers as
\[\B\{\frac{q(q-1)}{2}+1+pq+\frac{p(p+1)}{2}: p=0,1,\ldots,\;q=1,2,\ldots\B\}.\]
In fact, a natural number lying in row $p$ and column $q$ ($p=0,1,\ldots,\;q=1,2,\ldots$) would be $\frac{q(q-1)}{2}+1+pq+\frac{p(p+1)}{2}$. For $n\in\N,$ let $r$ be the least positive integer such that $\frac{r(r+1)}{2}\geq n$ and $s=\frac{r(r+1)}{2}-n.$ Then $n$ lies in row $n_r=r-s-1$ and column $n_c=s+1.$ Thus without any loss of generality we may denote the set $G_n$ by its place position $G_{n_r,n_c}$ say, or more simply by $G_{i,j}$ for suitable $i,j$  and $G_{i,j}$ may be denoted by $G_n$ for suitable $n,$ and similarly for other terms.
 We can write $G_k=G_{p,q}$ for suitable $p,q.$ Using the continuity of $e^z,$ for each $k=1,2,\ldots,$ choose $\eta_{p,q}$ and $\xi_{p,q}$  such that
\begin{align*}
|e^w+(4(\frac{q(q+1)}{2}+1+p(q+1)+\frac{p(p+1)}{2})+2)|<\frac{1}{2}, \; \mbox{whenever}\\      
 |w-(\pi i+\log(4(\frac{q(q+1)}{2}+1+p(q+1)+\frac{p(p+1)}{2})+2))|<\eta_{p,q},
\end{align*} and
 
\begin{align*}
|e^w-(4(\frac{q(q+1)}{2}+1+p(q+1)+\frac{p(p+1)}{2})+2)|<\frac{1}{2},\;  \mbox{whenever}\\ 
|w-\log(4(\frac{q(q+1)}{2}+1+p(q+1)+\frac{p(p+1)}{2})+2)|<\xi_{p,q}.
\end{align*}
Also choose $\delta_0,\delta_q,\delta_{q}'$ so that 
\[|e^w-2|<\frac{1}{2},\; \mbox{whenever}\; |w-\log2|<\delta_0,\]
\[|e^w+(4(\frac{q(q-1)}{2}+1)+2)|<\frac{1}{2}, \;\mbox{whenever}\; |w-(\pi i+\log(4(\frac{q(q-1)}{2}+1)+2))|<\delta_q,\] and
\[|e^w-(4(\frac{q(q-1)}{2}+1)+2)|<\frac{1}{2}, \;\mbox{whenever}\; |w- \log(4(\frac{q(q-1)}{2}+1)+2)|<\delta_q'.\]
Define
\begin{equation}\notag
\alpha(z)=
\begin{cases}
\log2, & z\in G_0\cup\B\{{\bigcup}_{k=1}^{\ity}( L_k\cup M_k)\B\}\\
\pi i+\log(4(\frac{q(q-1)}{2}+1)+2), & z\in G_{p,q},\; p\geq 0, q\geq 1\\
\pi i+\log(4(\frac{q(q+1)}{2}+1+p(q+1)+\frac{p(p+1)}{2})+2), & z\in B_{p,q},\;  p\geq 0, q\geq 1\\

\end{cases}
\end{equation}
\begin{equation}\notag
\beta(z)=
\begin{cases}
\log2, & z\in G_0\cup\B\{{\bigcup}_{k=1}^{\ity}( L_k\cup M_k)\B\}\\
\log(4(\frac{q(q-1)}{2}+1)+2),& z\in B_{p,q},\;p=0,1,\ldots, q=1,2,\ldots\\
\log(4(\frac{q(q+1)}{2}+1+p(q+1)+\frac{p(p+1)}{2})+2), & z\in G_{p,q},\;p=0,1,\ldots, q=1,2,\ldots
\end{cases}
\end{equation}
\begin{equation}\notag
\epsilon(z)=
\begin{cases}
\delta_0, & z\in G_0\cup\B{\{{\bigcup}_{k=1}^{\ity}( L_k\cup M_k)\B\}}\\
\delta_q, & z\in G_{p,q},\;p=0,1,\ldots, q=1,2,\ldots\\
\eta_{p,q}, & z\in B_{p,q},\;p=0,1,\ldots,q=1,2,\ldots
\end{cases}
\end{equation}
and
\begin{equation}\notag
\epsilon_{1}(z)=
\begin{cases}
\delta_0, & z\in G_0\cup\B{\{{\bigcup}_{k=1}^{\ity}( L_k\cup M_k)\B\}}\\
\delta_q', & z\in B_{p,q},\;p=0,1,\ldots, q=1,2,\ldots\\
\xi_{p,q}, & z\in G_{p,q},\;p=0,1,\ldots,q=1,2,\ldots
\end{cases}
\end{equation}
Clearly $\alpha(z)$ is continuous on $S$ and analytic in $S^\circ$. Thus there is an entire function $\gamma(z)$ such that $|\gamma(z)-\alpha(z)|<\epsilon(z)$ for all $z\in S.$ The function $f(z)=e^{\gamma(z)}$ is an entire function which satisfies
\[|f(z)-2|<\frac{1}{2},\quad z\in G_0\cup\B{\{{\bigcup}_{k=1}^{\ity}( L_k\cup M_k)\B\}},\]
\[|f(z)+(4(\frac{q(q-1)}{2}+1)+2)|<\frac{1}{2},\quad z\in G_{p,q},\;p=0,1,\ldots,q=1,2,\ldots,\] and 
\[|f(z)+(4(\frac{q(q+1)}{2}+1+p(q+1)+\frac{p(p+1)}{2})+2)|<\frac{1}{2},\quad z\in B_{p,q},\;p=0,1,\ldots,q=1,2,\ldots.\]
Therefore $f$ maps $G_0\cup\B{\{{\bigcup}_{k=1}^{\ity}( L_k\cup M_k)\B\}}$ into the smaller disk $|z-2|<\frac{1}{2}$ of $G_0,$ and so  $G_0\cup\B{\{{\bigcup}_{k=1}^{\ity}( L_k\cup M_k)\B\}}$ contains a fixed point $\zeta$ such that $f^{n}\B(G_0\cup\B{\{{\bigcup}_{k=1}^{\ity}( L_k\cup M_k)\B\}}\B)\to\zeta$ as $n\to\ity,$ which implies that  $G_0\cup\B{\{{\bigcup}_{k=1}^{\ity}( L_k\cup M_k)\B\}}\subset F(f).$ Also $f$ maps each $G_{p,q},\;p=0,1,\ldots, q=1,2,\ldots$ inside a smaller disk of $B_{0,q}$ and each $B_{p,q},\;p=0,1,\ldots, q=1,2,\ldots$ is mapped inside a smaller disk of $B_{p,q+1}.$ Therefore each $G_{p,q}$ and also each $B_{p,q}$ lies in the Fatou set of $f$. Let $C_{p,q}$ be the components of $G_{p,q}$. Then $C_{p,q}$ are wandering domains for $f.$ 
Now as $\beta(z)$ is also  continuous  on $S$ and analytic in $S^\circ$, there exist an entire function $\gamma_{1}(z)$ such that if $g(z)=e^{\gamma_{1}(z)}$ then
 \[|g(z)-2|<\frac{1}{2},\quad z\in G_0\cup\B{\{{\bigcup}_{k=1}^{\ity}( L_k\cup M_k)\B\}},\]
\[|g(z)-(4(\frac{q(q-1)}{2}+1)+2)|<\frac{1}{2},\quad z\in B_{p,q},\;p=0,1,\ldots, q=1,2,\ldots,\]  and
\[|g(z)-(4(\frac{q(q+1)}{2}+1+p(q+1)+\frac{p(p+1)}{2})+2)|<\frac{1}{2},\quad  z\in G_{p,q},\;p=0,1,\ldots,q=1,2,\ldots.\]
Using similar arguments as before, $G_0\cup\B{\{{\bigcup}_{k=1}^{\ity}( L_k\cup M_k)\B\}}\subset F(g).$ Also $g$ maps  each $B_{p,q},\;p=0,1,\ldots,q=1,2,\ldots$ inside a smaller disk in $G_{0,q}$ and each $G_{p,q},\;p=0,1,\ldots,q=1,2,\ldots$ is mapped inside a smaller disk in $G_{p,q+1},$ so that each $G_{p,q}$ as well as each $B_{p,q}$ lies in $F(g).$
Let $C_{0,q}, C_{0,q}'$ respectively be the components of $F(f)$ and $F(g)$ containing $G_{0,q}.$ Since every unbounded Fatou component of a transcendental entire function is simply connected \cite[Theorem 1]{baker1}, it follows that $G_{0,q}=C_{0,q}\cap C_{0,q}'\cap G_{0,q}$ are simply connected domains  which lies in the wandering component of $f$ and also in the wandering component of $g$. Also for any $z\in G_{0,q},\;|f(z)+(4(\frac{q(q-1)}{2}+1)+2)|<\frac{1}{2},$ so that $f(z)\in B_{0,q}$ and consequently, $|g(f(z))-4(\frac{q(q-1)}{2}+1)+2)|<\frac{1}{2}.$ Thus $g(f(z))$ lies in $G_{0,q}.$ Further since $f$ and $g$ are contractions,  it follows  that $G_{0,q}\subset F(g\circ f)$ and lies in periodic component of $g\circ f.$
\end{proof}
\begin{theorem}\label{sec2,thm2}
There exist transcendental entire functions $f$ and $g$ having infinite number of domains which lies in the wandering component of $f$ and  wandering component of $g$ and also lies in the wandering component of $f\circ g$ and the wandering component of $g\circ f.$
\end{theorem}
\begin{proof}
Let $S,\alpha,f,\epsilon(z),\delta_0,\delta_q,\eta_{p,q}$ be as defined in the Theorem \ref{sec2,thm1}. For each $k=1,2,\ldots,$ choose $\xi_{p,q}$ so that
\begin{align*}
|e^w-(4(\frac{q(q+1)}{2}+1+p(q+1)+\frac{p(p+1)}{2})+2)|<\frac{1}{2}, \;\mbox{whenever}\\  
|w-\log(4(\frac{q(q+1)}{2}+1+p(q+1)+\frac{p(p+1)}{2})+2)|<\xi_{p,q}.
\end{align*}
Define
\begin{equation}\notag
\beta(z)=
\begin{cases}
\log2, & z\in G_0\cup\B{\{{\bigcup}_{k=1}^{\ity}( L_k\cup M_k)\B\}}\\
\pi i+\log(4(\frac{q(q+1)}{2}+1+p(q+1)+\frac{p(p+1)}{2})+2), & z\in B_{p,q}, \;p\geq 0, q\geq 1\\

\pi i+\log(4(\frac{q(q-1)}{2}+1)+2), & z\in G_{0,q},\;q=1,2,\ldots\\
\log(4(\frac{q(q+1)}{2}+1+p(q+1)+\frac{p(p+1)}{2})+2), & z\in G_{p,q},\;p,q=1,2,\ldots\\
\end{cases}
\end{equation}
and
\begin{equation}\notag
\epsilon_{1}(z)=
\begin{cases}
\delta_0, & z\in G_0\cup\B{\{{\bigcup}_{k=1}^{\ity}( L_k\cup M_k)\B\}}\\
\eta_{p,q}, & z\in B_{p,q},\;p=0,1,\ldots,q=1,2,\ldots\\
\delta_q, & z\in G_{0,q},\;q=1,2,\ldots\\
\xi_{p,q}, & z\in G_{p,q},\;p,q=1,2,\ldots\\
\end{cases}
\end{equation}
Clearly $\beta(z)$ is continuous on $S$ and analytic in $S^\circ.$ So there is an entire function $\gamma_1(z)$ such that $|\beta(z)-\gamma_{1}(z)|<\epsilon_1(z)$ for all $z\in S.$ The entire function $g(z)=e^{\gamma_1(z)}$ satisfies
\[|g(z)-2|<\frac{1}{2},\quad z\in G_0\cup\B{\{{\bigcup}_{k=1}^{\ity}( L_k\cup M_k)\B\}},\]
\[|g(z)+(4(\frac{q(q+1)}{2}+1+p(q+1)+\frac{p(p+1)}{2})+2)|<\frac{1}{2},\quad z\in B_{p,q},\;p=0,1,\ldots,q=1,2,\ldots,\]
\[|g(z)+(4(\frac{q(q-1)}{2}+1)+2)|<\frac{1}{2},\quad z\in G_{0,q},\;q=1,2,\ldots,\] and
\[|g(z)-(4(\frac{q(q+1)}{2}+1+p(q+1)+\frac{p(p+1)}{2})+2)|<\frac{1}{2},\quad z\in G_{p,q},\;p,q=1,2,\ldots.\]
This implies that $g$ maps $G_{0,q},\;q=1,2,\ldots$ into a smaller disk in $B_{0,q},$ each $G_{p,q},\;p,q=1,2,\ldots$ into a smaller disk in $G_{p,q+1},$ and each $B_{p,q},\;p=0,1,\ldots, q=1,2,\ldots$ into a smaller disk in $B_{p,q+1}.$ As in Theorem \ref{sec2,thm1}, it can be easily seen that $G_{0,q}$ lies in the wandering component of $f$ and wandering component of $g$ and also lies in the wandering component of $f\circ g$ and wandering component of $g\circ f.$
\end{proof}
\begin{theorem}\label{sec2,thm3}
There exist transcendental entire functions $f$ and $g$ having infinite number of domains which lies in the periodic component of $f$ and periodic component of $g$ but lies  in the wandering component of $g\circ f$ and the wandering component of $f\circ g.$
\end{theorem}
\begin{proof}
We construct the functions $f$ and $g$ on the Carleman set $S$ of Theorem \ref{sec2,thm1}.
 Using the continuity of $e^z,$ for each $q=1,2,\ldots,$ choose $\delta_0,\delta_q,\delta_{q}'$ so that 
\[|e^w-2|<\frac{1}{2},\;\mbox{whenever}\; |w-\log2|<\delta_0,\]
\[|e^w+(4(\frac{q(q-1)}{2}+1)+2)|<\frac{1}{2}, \;
\mbox{whenever}\; |w-(\pi i+\log(4(\frac{q(q-1)}{2}+1)+2))|<\delta_q,\] and 
\[|e^w-(4(\frac{q(q-1)}{2}+1)+2)|<\frac{1}{2},\;
\mbox{whenever}\; |w- \log(4(\frac{q(q-1)}{2}+1)+2)|<\delta_q'.\]
Also choose $\xi_{p,q}$ and $\eta_{p,q}$  such that
\begin{align*}
|e^w-(4(\frac{q(q+1)}{2}+1+p(q+1)+\frac{p(p+1)}{2})+2)|<\frac{1}{2},\; \mbox{whenever}\\
  |w-\log(4(\frac{q(q+1)}{2}+1+p(q+1)+\frac{p(p+1)}{2})+2)|<\xi_{p,q},
\end{align*} and
 
\begin{align*}
|e^w+(4(\frac{q(q+1)}{2}+1+p(q+1)+\frac{p(p+1)}{2})+2)|<\frac{1}{2},\;\mbox{whenever}\\
  |w-(\pi i+\log(4(\frac{q(q+1)}{2}+1+p(q+1)+\frac{p(p+1)}{2})+2))|<\eta_{p,q}.
\end{align*}
Define
\begin{equation}\notag
\alpha(z)=
\begin{cases}
\log2, & z\in G_0\cup\B{\{{\bigcup}_{k=1}^{\ity}( L_k\cup M_k)\B\}}\\
\pi i+\log(4(\frac{q(q-1)}{2}+1)+2), & z\in G_{0,q},\;q=1,2,\ldots\\
\log(4(\frac{q(q-1)}{2}+1)+2), & z\in B_{0,q},\;q=1,2,\ldots\\
\pi i+\log(4(\frac{q(q+1)}{2}+1+p(q+1)+\frac{p(p+1)}{2})+2), & z\in B_{p,q},\;p,q=1,2,\ldots\\
\log(4(\frac{q(q+1)}{2}+1+p(q+1)+\frac{p(p+1)}{2})+2), & z\in G_{p,q},\;p,q=1,2,\ldots
\end{cases}
\end{equation}
and
\begin{equation}\notag
\epsilon(z)=
\begin{cases}
\delta_0, & z\in G_0\cup\B{\{{\bigcup}_{k=1}^{\ity}( L_k\cup M_k)\B\}}\\
\delta_q, & z\in G_{0,q},\;q=1,2,\ldots\\
\delta_q', & z\in B_{0,q},\;q=1,2,\ldots\\
\eta_{p,q}, & z\in B_{p,q},\;p,q=1,2,\ldots\\
\xi_{p,q}, & z\in G_{p,q},\; p,q=1,2,\ldots
\end{cases}
\end{equation}
As $S$ is a Carleman set, $\alpha(z)$ is continuous on $S$ and analytic in $S^\circ$, there exist an entire function $\gamma(z)$ such that $|\gamma(z)-\alpha(z)|<\epsilon(z)$ for all $z\in S.$ Consequently, the entire function $f(z)=e^{\gamma(z)}$ satisfies
\[|f(z)-2|<\frac{1}{2},\quad z\in G_0\cup\B{\{{\bigcup}_{k=1}^{\ity}( L_k\cup M_k)\B\}},\]
\[|f(z)+(4(\frac{q(q-1)}{2}+1)+2)|<\frac{1}{2},\quad z\in G_{0,q},\;q=1,2,\ldots,\]
\[|f(z)-(4(\frac{q(q-1)}{2}+1)+2)|<\frac{1}{2},\quad z\in B_{0,q},\;q=1,2,\ldots,\]
\[|f(z)+(4(\frac{q(q+1)}{2}+1+p(q+1)+\frac{p(p+1)}{2})+2)|<\frac{1}{2},\quad z\in B_{p,q},\;p,q=1,2,\ldots,\] and
\[|f(z)-(4(\frac{q(q+1)}{2}+1+p(q+1)+\frac{p(p+1)}{2})+2)|<\frac{1}{2},\quad z\in G_{p,q},\;p,q=1,2,\ldots.\]
Therefore the function $f$ satisfies
\[f\B( G_0\cup\B{\{{\bigcup}_{k=1}^{\ity}( L_k\cup M_k)\B\}}\B)\subset \{z:|z-2|<\frac{1}{2}\}\subset G_0,\]
\[f(G_{0,q})\subset\{z:|z+(4(\frac{q(q-1)}{2}+1)+2)|<\frac{1}{2}\}\subset B_{0,q},\;q=1,2,\ldots,\]
\[f(B_{0,q})\subset\{z:|z-(4(\frac{q(q-1)}{2}+1)+2)|<\frac{1}{2}\}\subset G_{0,q},\;q=1,2,\ldots,\]
\[f(B_{p,q})\subset\{z:|z+(4(\frac{q(q+1)}{2}+1+p(q+1)+\frac{p(p+1)}{2})+2)|<\frac{1}{2}\}\subset B_{p,q+1},\; p,q\geq 1,\] and
\[f(G_{p,q})\subset\{z:|z-(4(\frac{q(q+1)}{2}+1+p(q+1)+\frac{p(p+1)}{2})+2)|<\frac{1}{2}\}\subset G_{p,q+1},\; p,q\geq 1.\]
We now choose $\mu_q,\mu_q'$ so that
\[|e^w+(4(\frac{q(q-1)}{2}+q+2)+2)|<\frac{1}{2},\;\mbox{whenever}\; |w-(\pi i+\log(4(\frac{q(q-1)}{2}+q+2)+2))|<\mu_q,\] and 
\[|e^w+(4(\frac{q(q-1)}{2}+2q+4)+2)|<\frac{1}{2},\;\mbox{whenever}\; |w-(\pi i+\log(4(\frac{q(q-1)}{2}+2q+4)+2))|<\mu_q'.\]
Next define
\begin{equation}\notag
\beta(z)=
\begin{cases}
\log2, & z\in G_0\cup\B{\{{\bigcup}_{k=1}^{\ity}( L_k\cup M_k)\B\}}\\
\pi i+\log(4(\frac{q(q-1)}{2}+q+2)+2), & z\in G_{0,q},\;q=1,2,\ldots\\
\pi i+\log(4(\frac{q(q-1)}{2}+2q+4)+2), & z\in B_{0,q},\;q=1,2,\ldots\\
\pi i+\log(4(\frac{q(q+1)}{2}+1+p(q+1)+\frac{p(p+1)}{2})+2), & z\in B_{p,q},\;p\geq 2, q=1,2,\ldots\\
\log(4(\frac{q(q-1)}{2}+1)+2),& z\in B_{1,q},\;q=1,2,\ldots\\
\log(4(\frac{q(q+1)}{2}+1+p(q+1)+\frac{p(p+1)}{2})+2), & z\in G_{p,q},\;p,q=1,2,\ldots
\end{cases}
\end{equation}
and
\begin{equation}\notag
\epsilon{_1}(z)=
\begin{cases}
\delta_0, & z\in G_0\cup\B{\{{\bigcup}_{k=1}^{\ity}( L_k\cup M_k)\B\}}\\
\mu_q, & z\in G_{0,q},\;q=1,2,\ldots\\
\mu_q', & z\in B_{0,q},\;q=1,2,\ldots\\
\eta_{p,q}, & z\in B_{p,q},\;p\geq 2, q=1,2,\ldots\\
\delta_q', & z\in B_{1,q},\;q=1,2,\ldots\\
\xi_{p,q}, & z\in G_{p,q},\; p,q=1,2,\ldots
\end{cases}
\end{equation}
As $\beta(z)$ is continuous on $S$ and analytic in $S^\circ$, there exist an entire function $\gamma_1(z)$ such that $|\gamma_1(z)-\beta(z)|<\epsilon_{1}(z)$ for all $z\in S.$ The entire function $g(z)=e^{\gamma_1(z)}$ satisfies
\[|g(z)-2|<\frac{1}{2},\quad z\in G_0\cup\B{\{{\bigcup}_{k=1}^{\ity}( L_k\cup M_k)\B\}},\]
\[|g(z)+(4(\frac{q(q-1)}{2}+q+2)+2)|<\frac{1}{2},\quad  z\in G_{0,q},\;q=1,2,\ldots,\]
\[|g(z)+(4(\frac{q(q-1)}{2}+2q+4)+2)|<\frac{1}{2}, \quad z\in B_{0,q},\;q=1,2,\ldots,\]
\[|g(z)+(4(\frac{q(q+1)}{2}+1+p(q+1)+\frac{p(p+1)}{2})+2)|<\frac{1}{2},\quad  z\in B_{p,q},\;p\geq 2, q=1,2,\ldots,\]
\[|g(z)-(4(\frac{q(q-1)}{2}+1)+2)|<\frac{1}{2},\quad z\in B_{1,q},\;q=1,2,\ldots,\] and
\[|g(z)-(4(\frac{q(q+1)}{2}+1+p(q+1)+\frac{p(p+1)}{2})+2|<\frac{1}{2},\quad  z\in G_{p,q},\;p,q=1,2,\ldots.\]
Thus the function $g$ satisfies
\[g\B(G_0\cup\B{\{{\bigcup}_{k=1}^{\ity}( L_k\cup M_k)\B\}}\B)\subset\{z:|z-2|<\frac{1}{2}\}\subset G_0,\]
\[g(G_{0,q})\subset\{z:|z+(4(\frac{q(q-1)}{2}+q+2)+2)|<\frac{1}{2}\}\subset B_{1,q},\;q=1,2,\ldots,\]
\[g(B_{0,q})\subset\{z:|z+(4(\frac{q(q-1)}{2}+2q+4)+2)|<\frac{1}{2}\}\subset B_{2,q},\;q=1,2,\ldots,\]
\[g(B_{1,q})\subset\{z:|z-(4(\frac{q(q-1)}{2}+1)+2)|<\frac{1}{2}\}\subset G_{0,q},\;q=1,2,\ldots,\]
\[g(B_{p,q})\subset\{z:|z+(4(\frac{q(q+1)}{2}+1+p(q+1)+\frac{p(p+1)}{2})+2)|<\frac{1}{2}\}\subset B_{p,q+1},\;p\geq 2, q\geq 1,\] and 
\[g(G_{p,q})\subset\{z:|z-(4(\frac{q(q+1)}{2}+1+p(q+1)+\frac{p(p+1)}{2})+2)|<\frac{1}{2}\}\subset G_{p,q+1,}\;p, q=1,2,\ldots.\]
It can now be easily seen that $G_{0,q}$ lies in the periodic component of $f$ and periodic component of $g$ but lies in the wandering component of $g\circ f$ and wandering component of $f\circ g.$
\end{proof}
\begin{theorem}\label{sec2,thm4}
There exist transcendental entire functions $f$ and $g$ having infinite number of domains which lies in the periodic component of $f$ and periodic component of $g$ and also in the periodic component of $g\circ f$ but lies in the wandering component of $f\circ g.$
\end{theorem}
\begin{proof}
Let $S,\delta_0,\delta_q,\delta_q'$ be as in Theorem \ref{sec2,thm1}. Choose $\delta_q'',\delta_{1,q},\delta_{2,q}$ such that
\[|e^w+(4(\frac{q(q-1)}{2}+3q+7)+2)|<\frac{1}{2},\;\mbox{whenever}\; |w-(\pi i+\log(4(\frac{q(q-1)}{2}+3q+7)+2))|<\delta_{1,q},\]
\[|e^w+(4(\frac{q(q-1)}{2}+q+2)+2)|<\frac{1}{2},\;\mbox{whenever}\; |w-(\pi i+\log(4(\frac{q(q-1)}{2}+q+2)+2))|<\delta_{2,q},\] and 
\[|e^w+(4(\frac{q(q-1)}{2}+2q+4)+2)|<\frac{1}{2},\;\mbox{whenever}\; |w-(\pi i+\log(4(\frac{q(q-1)}{2}+2q+4)+2))|<\delta_q''.\]
Also choose $\xi_{p,q}$ and $\eta_{p,q}$ such that
\begin{align*}
|e^w-(4(\frac{q(q+1)}{2}+1+p(q+1)+\frac{p(p+1)}{2})+2)|<\frac{1}{2},\;\mbox{whenever}\\
  |w-\log(4(\frac{q(q+1)}{2}+1+p(q+1)+\frac{p(p+1)}{2})+2)|<\xi_{p,q},
\end{align*} and
\begin{align*}
|e^w+|(4(\frac{q(q+1)}{2}+1+p(q+1)+\frac{p(p+1)}{2})+2)|<\frac{1}{2},\;\mbox{whenever}\\
 |w-(\pi i+\log(4(\frac{q(q+1)}{2}+1+p(q+1)+\frac{p(p+1)}{2})+2)|<\eta_{p,q}.
\end{align*}
Define
\begin{equation}\notag
\alpha(z)=
\begin{cases}
\log2, & z\in G_0\cup\B{\{{\bigcup}_{k=1}^{\ity}( L_k\cup M_k)\B\}}\\
\pi i+\log(4(\frac{q(q-1)}{2}+1)+2), & z\in G_{0,q},\;q=1,2,\ldots\\
\log(4(\frac{q(q-1)}{2}+1)+2), & z\in B_{0,q},\;q=1,2,\ldots\\
\pi i+\log(4(\frac{q(q-1)}{2}+3q+7)+2), & z\in B_{1,q},\;q=1,2,\ldots\\
\pi i+\log(4(\frac{q(q-1)}{2}+q+2)+2), & z\in B_{2,q},\;q=1,2,\ldots\\
\pi i+\log(4(\frac{q(q+1)}{2}+1+p(q+1)+\frac{p(p+1)}{2})+2), & z\in B_{p,q},\;p\geq 3,q=1,2,\ldots\\
\log(4(\frac{q(q+1)}{2}+1+p(q+1)+\frac{p(p+1)}{2})+2), & z\in G_{p,q},\;p,q=1,2,\ldots
\end{cases}
\end{equation}
and
\begin{equation}\notag
\epsilon(z)=
\begin{cases}
\delta_0, & z\in G_0\cup\B{\{{\bigcup}_{k=1}^{\ity}( L_k\cup M_k)\B\}}\\
\delta_q, & z\in G_{0,q},\;q=1,2,\ldots\\
\delta_q', & z\in B_{0,q},\;q=1,2,\ldots\\
\delta_{1,q}, & z\in B_{1,q},\;q=1,2,\ldots\\
\delta_{2,q}, & z\in B_{2,q},\;q=1,2,\ldots\\
\eta_{p,q}, & z\in B_{p,q},\;p\geq 3,q=1,2,\ldots\\
\xi_{p,q}, & z\in G_{p,q},\; p,q=1,2,\ldots
\end{cases}
\end{equation}
As in Theorem \ref{sec2,thm3}, there exist an entire function $\gamma(z)$ and consequently, an entire function $f(z)=e^{\gamma(z)}$ which satisfies
\[|f(z)-2|<\frac{1}{2},\quad z\in G_0\cup\B{\{{\bigcup}_{k=1}^{\ity}( L_k\cup M_k)\B\}},\]
\[|f(z)+(4(\frac{q(q-1)}{2}+1)+2)|<\frac{1}{2},\quad z\in G_{0,q},\;q=1,2,\ldots,\]
\[|f(z)-(4(\frac{q(q-1)}{2}+1)+2)|<\frac{1}{2},\quad z\in B_{0,q},\;q=1,2,\ldots,\]
\[|f(z)+(4(\frac{q(q-1)}{2}+3q+7)+2)|<\frac{1}{2},\quad z\in B_{1,q},\;q=1,2,\ldots,\]
\[|f(z)+(4(\frac{q(q-1)}{2}+q+2)+2)|<\frac{1}{2},\quad z\in B_{2,q},\;q=1,2,\ldots,\]
\[|f(z)+(4(\frac{q(q+1)}{2}+1+p(q+1)+\frac{p(p+1)}{2})+2)|<\frac{1}{2},\quad z\in B_{p,q},\;p\geq 3,q=1,2,\ldots,\] and 
\[|f(z)-(4(\frac{q(q+1)}{2}+1+p(q+1)+\frac{p(p+1)}{2})+2)|<\frac{1}{2},\quad z\in G_{p,q},\;p,q=1,2,\ldots.\]
 So the function $f$ satisfies
\[f\B(G_0\cup\B{\{{\bigcup}_{k=1}^{\ity}( L_k\cup M_k)\B\}}\B)\subset \{z:|z-2|<\frac{1}{2}\}\subset G_0,\]
\[f(G_{0,q})\subset\{z:|z+(4(\frac{q(q-1)}{2}+1)+2)|<\frac{1}{2}\}\subset B_{0,q},\;q=1,2,\ldots,\]
\[f(B_{0,q})\subset\{z:|z-(4(\frac{q(q-1)}{2}+1)+2)|<\frac{1}{2}\}\subset G_{0,q},\;q=1,2,\ldots,\]
\[f(B_{1,q})\subset\{z:|z+(4(\frac{q(q-1)}{2}+3q+7)+2)|<\frac{1}{2}\}\subset B_{3,q},\;q=1,2,\ldots,\]
\[f(B_{2,q})\subset\{z:|z+(4(\frac{q(q-1)}{2}+q+2)+2)|<\frac{1}{2}\}\subset B_{1,q},\;q=1,2,\ldots,\]
\[f(B_{p,q})\subset\{z:|z+(4(\frac{q(q+1)}{2}+1+p(q+1)+\frac{p(p+1)}{2})+2)|<\frac{1}{2}\}\subset B_{p,q+1},\;p\geq 3,q\geq 1,\] and 
\[f(G_{p,q})\subset\{z:|z-(4(\frac{q(q+1)}{2}+1+p(q+1)+\frac{p(p+1)}{2})+2|<\frac{1}{2}\}\subset G_{p,q+1},\;p,q=1,2,\ldots.\]
We now define
\begin{equation}\notag
\beta(z)=
\begin{cases}
\log2, & z\in G_0\cup\B{\{{\bigcup}_{k=1}^{\ity}( L_k\cup M_k)\B\}}\\
\pi i+\log(4(\frac{q(q-1)}{2}+q+2)+2), & z\in G_{0,q},\;q=1,2,\ldots\\
\pi i+\log(4(\frac{q(q-1)}{2}+2q+4)+2), & z\in B_{0,q},\;q=1,2,\ldots\\
\log(4(\frac{q(q-1)}{2}+1)+2), & z\in B_{1,q},\;q=1,2,\ldots\\
\pi i+\log(4(\frac{q(q+1)}{2}+1+p(q+1)+\frac{p(p+1)}{2})+2), & z\in B_{p,q},\;p\geq 2, q=1,2,\ldots\\

\log(4(\frac{q(q+1)}{2}+1+p(q+1)+\frac{p(p+1)}{2})+2, & z\in G_{p,q},\;p,q=1,2,\ldots
\end{cases}
\end{equation}
and
\begin{equation}\notag
\epsilon{_1}(z)=
\begin{cases}
\delta_0, & z\in G_0\cup\B{\{{\bigcup}_{k=1}^{\ity}( L_k\cup M_k)\B\}}\\
\delta_{2,q}, & z\in G_{0,q},\;q=1,2,\ldots\\
\delta_q'', & z\in B_{0,q},\;q=1,2,\ldots\\
\delta_q', & z\in B_{1,q},\;q=1,2,\ldots\\
\eta_{p,q}, & z\in B_{p,q},\;p\geq 2, q=1,2,\ldots\\
\xi_{p,q}, & z\in G_{p,q},\; p,q=1,2,\ldots
\end{cases}
\end{equation}
Again as in Theorem \ref{sec2,thm3}, there exist an entire function $g(z)=e^{\gamma_1(z)}$ which satisfies

\[|g(z)-2|<\frac{1}{2},\quad z\in G_0\cup\B{\{{\bigcup}_{k=1}^{\ity}( L_k\cup M_k)\B\}},\]
\[|g(z)+(4(\frac{q(q-1)}{2}+q+2)+2)|<\frac{1}{2},\quad  z\in G_{0,q},\;q=1,2,\ldots,\]
\[|g(z)+(4(\frac{q(q-1)}{2}+2q+4)+2)|<\frac{1}{2}, \quad z\in B_{0,q},\;q=1,2,\ldots,\]
\[|g(z)-(4(\frac{q(q-1)}{2}+1)+2)|<\frac{1}{2},\quad z\in B_{1,q},\;q=1,2,\ldots,\]
\[|g(z)+(4(\frac{q(q+1)}{2}+1+p(q+1)+\frac{p(p+1)}{2})+2)|<\frac{1}{2},\quad  z\in B_{p,q},\;p\geq 2, q=1,2,\ldots,\] and 
\[|g(z)-(4(\frac{q(q+1)}{2}+1+p(q+1)+\frac{p(p+1)}{2})+2)|<\frac{1}{2},\quad  z\in G_{p,q},\;p,q=1,2,\ldots.\]
Therefore the function $g$ satisfies
\[g\B( G_0\cup\B{\{{\bigcup}_{k=1}^{\ity}( L_k\cup M_k)\B\}}\B)\subset\{z:|z-2|<\frac{1}{2}\}\subset G_0,\]
\[g(G_{0,q})\subset\{z:|z+(4(\frac{q(q-1)}{2}+q+2)+2)|<\frac{1}{2}\}\subset B_{1,q},\;q=1,2,\ldots,\]
\[g(B_{0,q})\subset\{z:|z+(4(\frac{q(q-1)}{2}+2q+4)+2)|<\frac{1}{2}\}\subset B_{2,q},\;q=1,2,\ldots,\]
\[g(B_{1,q})\subset\{z:|z-(4(\frac{q(q-1)}{2}+1)+2)|<\frac{1}{2}\}\subset G_{0,q},\;q=1,2,\ldots,\]
\[g(B_{p,q})\subset\{z:|z+(4(\frac{q(q+1)}{2}+1+p(q+1)+\frac{p(p+1)}{2})+2)|<\frac{1}{2}\}\subset B_{p,q+1},\;p\geq 2, q\geq 1,\] and 
\[g(G_{p,q})\subset\{z:|z-(4(\frac{q(q+1)}{2}+1+p(q+1)+\frac{p(p+1)}{2})+2)|<\frac{1}{2}\}\subset G_{p,q+1,}\;p, q=1,2,\ldots.\]
It can now be easily seen that $G_{0,q}$ lies in the periodic component of $f$ and  periodic component of $g$ and also lies in the periodic component of $g\circ f$ but lies in the wandering component of $f\circ g.$
\end{proof}
\begin{theorem}\label{sec2,thm5}
There exist transcendental entire functions $f$ and $g$ having a domain which lies in the periodic component of $f$ and wandering component of $g$ and lies in the preperiodic component of $g\circ f$ and periodic component of $f\circ g.$
\end{theorem}
\begin{proof}
Let $S$ be the Carleman set defined as in Theorem \ref{sec2,thm1}.
Using the continuity of $e^z,$ for each $k=1,2,\ldots,$ choose $\eta_k$ and $\xi_k$ such that
\[|e^w+(4k+6)|<\frac{1}{2},\; \mbox{whenever}\; |w-(\pi i+\log(4k+6))|<\eta_k,\] and
\[|e^w-(4k+6)|<\frac{1}{2},\; \mbox{whenever}\; |w-\log(4k+6)|<\xi_k.\]
Also choose $\delta_0,\delta_1,\delta_2,\delta_3,\delta_4$ so that
\[|e^w-2|<\frac{1}{2},\; \mbox{whenever}\; |w-\log 2|<\delta_0,\]
\[|e^w+6|<\frac{1}{2},\; \mbox{whenever}\; |w-(\pi i+\log 6)|<\delta_1,\]
\[|e^w-6|<\frac{1}{2},\; \mbox{whenever}\; |w-\log 6|<\delta_2,\]
\[|e^w-10|<\frac{1}{2}, \;\mbox{whenever}\; |w-\log 10|<\delta_3,\] and
\[|e^w+10|<\frac{1}{2},\; \mbox{whenever}\; |w-(\pi i+\log 10)|<\delta_4.\]
We next define
\begin{equation}\notag
\alpha(z)=
\begin{cases}
\log 2,  & z\in G_0\cup\B{\{{\bigcup}_{k=1}^{\ity}( L_k\cup M_k)\B\}}\\
 \pi i+\log 6,   & z\in G_1\\
\log 6,   & z\in B_2\\
\pi i+\log 10,   & z\in G_2\\
\log 10,   & z\in B_1\\
 \pi i+\log(4k+6),   & z\in B_k,\;k=3,4,\ldots\\
\log(4k+6),   & z\in G_k,\;k=3,4,\ldots
\end{cases}
\end{equation}
\begin{equation}\notag
\beta(z)=
\begin{cases}
\log 2,  & z\in G_0\cup\B{\{{\bigcup}_{k=1}^{\ity}( L_k\cup M_k)\B\}}\\
 \pi i+\log 6,   & z\in G_1\\
\pi i+\log 10,   & z\in G_2\\
 \pi i+\log(4k+6),   & z\in B_k,\;k=1,2,\ldots\\
\log(4k+6),   & z\in G_k,\;k=3,4,\ldots
\end{cases}
\end{equation}
\begin{equation}\notag
\epsilon(z)=
\begin{cases}
\delta_0,   & z\in G_0\cup\B{\{{\bigcup}_{k=1}^{\ity}( L_k\cup M_k)\B\}}\\
\delta_1,   & z\in G_1\\
\delta_2,   & z\in B_2\\
\delta_4,   & z\in G_2\\
\delta_3,   & z\in B_1\\
 \eta_k,   & z\in B_k,\;k=3,4,\ldots\\
\xi_k,   & z\in G_k,\;k=3,4,\ldots
\end{cases}
\end{equation}
and
\begin{equation}\notag
\epsilon_1(z)=
\begin{cases}
\delta_0,  & z\in G_0\cup\B{\{{\bigcup}_{k=1}^{\ity}( L_k\cup M_k)\B\}}\\
 \delta_1,   & z\in G_1\\
\delta_4,   & z\in G_2\\
 \eta_k,   & z\in B_k,\;k=1,2,\ldots\\
\xi_k,   & z\in G_k,\;k=3,4,\ldots
\end{cases}
\end{equation}
Clearly $\alpha(z)$ is continuous on $S$ and analytic in $S^\circ.$ So there is an entire function $\gamma(z)$ such that $|\gamma(z)-\alpha(z)|<\epsilon(z)$ for all $z\in S.$ The function $f(z)=e^{\gamma(z)}$ is an entire function and it satisfies
\[|f(z)-2|<\frac{1}{2},\quad z\in G_0\cup\B{\{{\bigcup}_{k=1}^{\ity}( L_k\cup M_k)\B\}},\]
\[|f(z)+6|<\frac{1}{2},\quad z\in G_1,\]
\[|f(z)-6|<\frac{1}{2},\quad z\in B_2,\]
\[|f(z)+10|<\frac{1}{2},\quad z\in G_2,\]
\[|f(z)-10|<\frac{1}{2},\quad z\in B_1,\]
\[|f(z)+(4k+6)|<\frac{1}{2},\quad z\in B_k,\;k=3,4,\ldots,\] and
\[|f(z)-(4k+6)|<\frac{1}{2},\quad z\in G_k,\;k=3,4,\ldots.\]
Also $\beta(z)$ is  continuous  on $S$ and analytic in $S^\circ,$  there exist an entire function $\gamma_1(z)$  such that if $g(z)=e^{\gamma_1(z)},$ then
\[|g(z)-2|<\frac{1}{2},\quad z\in G_0\cup\B{\{{\bigcup}_{k=1}^{\ity}( L_k\cup M_k)\B\}},\]
\[|g(z)+6|<\frac{1}{2},\quad z\in G_1,\]
\[|g(z)+10|<\frac{1}{2},\quad z\in G_2,\]
\[|g(z)+(4k+6)|<\frac{1}{2},\quad z\in B_k,\;k=1,2,\ldots,\] and
\[|g(z)-(4k+6)|<\frac{1}{2},\quad z\in G_k,\;k=3,4,\ldots.\]
It can now be easily seen that $G_1$ lies in the periodic component of $f$ and wandering component of $g$ and lies in the preperiodic component of $g\circ f$ and periodic component of $f\circ g.$
\end{proof}
\begin{remark}\label{sec2,rem1}
\begin{enumerate}
\item[(i)] $B_2$ lies in the periodic component of $f$ and wandering component of $g$ and lies in the periodic component of $g\circ f$ and wandering component of $f\circ g.$
\item[(ii)] $B_1$ lies in the periodic component of $f$ and wandering component of $g$ and lies in the periodic component of $g\circ f$ and preperiodic component of $f\circ g.$
\item[(iii)] $G_2$ lies in the periodic component of $f$ and wandering component of $g$ and lies in the wandering component of $g\circ f$ and periodic component of $f\circ g.$
\item[(iv)] Both $G_k,\;k\geq 3$ and $B_k,\;k\geq 3$ lies in the wandering component of $f$ and wandering component of $g$ and also lies in the wandering component of $g\circ f$ and the wandering component of $f\circ g.$ 
\end{enumerate}
\end{remark}
\begin{theorem}\label{sec2,thm7}
There exist transcendental entire functions $f$ and $g$ having  a domain which lies in the periodic component of $f$ and wandering component of $g$ and also lies in the wandering component of $g\circ f$ and wandering component of $f\circ g.$
\end{theorem}
\begin{proof}
Let $S$ be the Carleman set defined as in Theorem \ref{sec2,thm1}.
Using the continuity of $e^z,$ for each $k=1,2,\ldots,$ choose $\delta_0,\delta_1,\delta_2,\delta_3,\eta_k,\xi_k$ such that
\[|e^w-2|<\frac{1}{2},\; \mbox{whenever}\; |w-\log 2|<\delta_0,\]
\[|e^w+6|<\frac{1}{2},\; \mbox{whenever}\; |w-(\pi i+\log 6)|<\delta_1,\]
\[|e^w+10|<\frac{1}{2}, \;\mbox{whenever}\; |w-(\pi i+\log 10)|<\delta_2,\]
\[|e^w-6|<\frac{1}{2}, \;\mbox{whenever}\; |w-\log 6|<\delta_3,\] 
\[|e^w+(4k+6)|<\frac{1}{2},\; \mbox{whenever}\; |w-(\pi i+\log(4k+6))|<\eta_k,\] and
\[|e^w-(4k+6)|<\frac{1}{2},\; \mbox{whenever}\; |w-\log(4k+6)|<\xi_k.\]
Define
\begin{equation}\notag
\alpha(z)=
\begin{cases}
\log 2,  & z\in G_0\cup\B{\{{\bigcup}_{k=1}^{\ity}( L_k\cup M_k)\B\}}\\
 \pi i+\log 6,   & z\in G_1\\
\pi i+\log 10,   & z\in B_1\\
\log 6,   & z\in B_2\\ 
 \pi i+\log(4k+6),   & z\in B_k,\;k=3,4,\ldots\\
\log(4k+6),   & z\in G_k,\;k=2,3,\ldots
\end{cases}
\end{equation}
\begin{equation}\notag
\beta(z)=
\begin{cases}
\log 2,  & z\in G_0\cup\B{\{{\bigcup}_{k=1}^{\ity}( L_k\cup M_k)\B\}}\\
 \pi i+\log 6,   & z\in G_1\\
 \pi i+\log(4k+6),   & z\in B_k,\;k=1,2,\ldots\\
\log(4k+6),   & z\in G_k,\;k=2,3,\ldots
\end{cases}
\end{equation}
\begin{equation}\notag
\epsilon(z)=
\begin{cases}
\delta_0,   & z\in G_0\cup\B{\{{\bigcup}_{k=1}^{\ity}( L_k\cup M_k)\B\}}\\
\delta_1,   & z\in G_1\\
\delta_2,   & z\in B_1\\
\delta_3,   & z\in B_2\\
\eta_k,   & z\in B_k,\;k=3,4,\ldots\\
\xi_k,   & z\in G_k,\;k=2,3,\ldots\\
\end{cases}
\end{equation}
and
\begin{equation}\notag
\epsilon_1(z)=
\begin{cases}
\delta_0,   & z\in G_0\cup\B{\{{\bigcup}_{k=1}^{\ity}( L_k\cup M_k)\B\}}\\
\delta_1,   & z\in G_1\\
\eta_k,   & z\in B_k,\;k=1,2,\ldots\\
\xi_k,   & z\in G_k,\;k=2,3,\ldots\\
\end{cases}
\end{equation}
Clearly  $\alpha(z)$ is continuous on $S$ and analytic in $S^\circ.$ Thus there is an entire function $\gamma(z)$ such that $|\gamma(z)-\alpha(z)|<\epsilon(z)$ for all $z\in S.$ The function $f(z)=e^{\gamma(z)}$ is an entire function and it satisfies
\[|f(z)-2|<\frac{1}{2},\quad z\in G_0\cup\B{\{{\bigcup}_{k=1}^{\ity}( L_k\cup M_k)\B\}},\]
\[|f(z)+6|<\frac{1}{2},\quad z\in G_1,\]
\[|f(z)+10|<\frac{1}{2},\quad z\in B_1,\]
\[|f(z)-6|<\frac{1}{2},\quad z\in B_2,\]
\[|f(z)+(4k+6)|<\frac{1}{2},\quad z\in B_k,\;k=3,4,\ldots,\] and
\[|f(z)-(4k+6)|<\frac{1}{2},\quad z\in G_k,\;k=2,3,\ldots.\]
Also $\beta(z)$ is  continuous  on $S$ and analytic in $S^\circ,$  there exist an entire function $\gamma_1(z)$  such that if $g(z)=e^{\gamma_1(z)}$ then
\[|g(z)-2|<\frac{1}{2},\quad z\in G_0\cup\B{\{{\bigcup}_{k=1}^{\ity}( L_k\cup M_k)\B\}},\]
\[|g(z)+6|<\frac{1}{2},\quad z\in G_1,\]
\[|g(z)+(4k+6)|<\frac{1}{2},\quad z\in B_k,\;k=1,2,\ldots,\] and
\[|g(z)-(4k+6)|<\frac{1}{2},\quad z\in G_k,\;k=2,3,\ldots.\]
It can now be easily seen that $G_1$ lies in the periodic component of $f$ and wandering component of $g$ and also lies in the wandering component of $g\circ f$ and wandering component of $f\circ g.$
\end{proof}
\begin{remark}\label{sec2,rem2}
\begin{enumerate}
\item[(i)] Both $B_1$ and $B_2$ lies in the periodic component of $f$ and wandering component of $g$ and also lies in the wandering component of $g\circ f$ and the wandering component of $f\circ g.$
\item[(ii)] Both $B_k,\;k\geq 3$ and $G_k,\;k\geq 2$ lies in the wandering component of $f$ and wandering component of $g$ and also lies in the wandering component of $g\circ f$ and the wandering component of $f\circ g.$
\end{enumerate}
\end{remark}
\begin{theorem}\label{sec2,thm9}
There exist transcendental entire functions $f$ and $g$ having a domain which lies in the periodic component of $f$ and wandering component of $g$ and  lies in the periodic component of $g\circ f$ and periodic component of $f\circ g.$
\end{theorem}
\begin{proof}
Let $S$ be the Carleman set defined as in Theorem \ref{sec2,thm1}.
Using the continuity of $e^z,$ for each $k=1,2,\ldots,$ choose $\eta_k,\xi_k,\mu_k$  so that
\[|e^w+(4k+6)|<\frac{1}{2},\; \mbox{whenever}\; |w-(\pi i+\log(4k+6))|<\eta_k,\]
\[|e^w-(4k+6)|<\frac{1}{2},\; \mbox{whenever}\; |w-\log(4k+6)|<\xi_k,\] and
\[|e^w+(4k-2)|<\frac{1}{2},\; \mbox{whenever}\; |w-(\pi i+\log(4k-2))|<\mu_k.\]
 Also choose $\delta_0,\delta_1,\delta_2,\delta_3$ so that
\[|e^w-2|<\frac{1}{2},\; \mbox{whenever}\; |w-\log 2|<\delta_0,\]
\[|e^w+6|<\frac{1}{2},\; \mbox{whenever}\; |w-(\pi i+\log 6)|<\delta_1,\]
\[|e^w-6|<\frac{1}{2},\; \mbox{whenever}\; |w-\log 6|<\delta_2,\] and
\[|e^w+14|<\frac{1}{2},\; \mbox{whenever}\; |w-(\pi i+\log 14)|<\delta_3.\]
We next define
\begin{equation}\notag
\alpha(z)=
\begin{cases}
\log 2,  & z\in G_0\cup\B{\{{\bigcup}_{k=1}^{\ity}( L_k\cup M_k)\B\}}\\
 \pi i+\log 6,   & z\in G_1\\
\log 6,   & z\in B_1\\ 
\log 6,   & z\in B_2\\ 
 \pi i+\log(4k-2),   & z\in B_k,\;k=3,4,\ldots\\
\log(4k+6),   & z\in G_k,\;k=2,3,\ldots
\end{cases}
\end{equation}
\begin{equation}\notag
\beta(z)=
\begin{cases}
\log 2,  & z\in G_0\cup\B{\{{\bigcup}_{k=1}^{\ity}( L_k\cup M_k)\B\}}\\
 \pi i+\log 6,   & z\in G_1\\
\log 6,   & z\in B_2\\ 
 \pi i+\log 14,   & z\in B_1\\
 \pi i+\log(4k+6),   & z\in B_k,\;k=3,4,\ldots\\
\log(4k+6),   & z\in G_k,\;k=2,3,\ldots
\end{cases}
\end{equation}
\begin{equation}\notag
\epsilon(z)=
\begin{cases}
\delta_0,   & z\in G_0\cup\B{\{{\bigcup}_{k=1}^{\ity}( L_k\cup M_k)\B\}}\\
\delta_1,   & z\in G_1\\
\delta_2,   & z\in B_1\\
\delta_2,   & z\in B_2\\
\mu_k,   & z\in B_k,\;k=3,4,\ldots\\
\xi_k,   & z\in G_k,\;k=2,3,\ldots\\
\end{cases}
\end{equation}
and
\begin{equation}\notag
\epsilon_1(z)=
\begin{cases}
\delta_0,   & z\in G_0\cup\B{\{{\bigcup}_{k=1}^{\ity}( L_k\cup M_k)\B\}}\\
\delta_1,   & z\in G_1\\
\delta_2,   & z\in B_2\\
\delta_3,   & z\in B_1\\
\eta_k,   & z\in B_k,\;k=3,4,\ldots\\
\xi_k,   & z\in G_k,\;k=2,3,\ldots\\
\end{cases}
\end{equation}
Clearly  $\alpha(z)$ is continuous on $S$ and analytic in $S^\circ.$ So there is an entire function $\gamma(z)$ such that $|\gamma(z)-\alpha(z)|<\epsilon(z)$ for all $z\in S.$ The function $f(z)=e^{\gamma(z)}$ is an entire function and it satisfies
\[|f(z)-2|<\frac{1}{2},\quad z\in G_0\cup\B{\{{\bigcup}_{k=1}^{\ity}( L_k\cup M_k)\B\}},\]
\[|f(z)+6|<\frac{1}{2},\quad z\in G_1,\]
\[|f(z)-6|<\frac{1}{2},\quad z\in B_1,\]
\[|f(z)-6|<\frac{1}{2},\quad z\in B_2,\]
\[|f(z)+(4k-2)|<\frac{1}{2},\quad z\in B_k,\;k=3,4,\ldots,\] and
\[|f(z)-(4k+6)|<\frac{1}{2},\quad z\in G_k,\;k=2,3,\ldots.\]
Also $\beta(z)$ is a continuous function on $S$ and analytic in $S^\circ,$  there exist an entire function $\gamma_1(z)$  such that if $g(z)=e^{\gamma_1(z)}$ then
\[|g(z)-2|<\frac{1}{2},\quad z\in G_0\cup\B{\{{\bigcup}_{k=1}^{\ity}( L_k\cup M_k)\B\}},\]
\[|g(z)+6|<\frac{1}{2},\quad z\in G_1,\]
\[|g(z)-6|<\frac{1}{2},\quad z\in B_2,\]
\[|g(z)+14|<\frac{1}{2},\quad z\in B_1,\]
\[|g(z)+(4k+6)|<\frac{1}{2},\quad z\in B_k,\;k=3,4,\ldots,\] and
\[|g(z)-(4k+6)|<\frac{1}{2},\quad z\in G_k,\;k=2,3,\ldots.\]
It can now be easily seen that $G_1$ lies in the periodic component of $f$ and wandering component of $g$ and  lies in the periodic component of $g\circ f$ and periodic component of $f\circ g.$
\end{proof}
\begin{remark}\label{sec2,rem3}
\begin{enumerate}
\item[(i)]  Similar to $G_1,$ $B_1$ lies in the periodic component of $f$ and wandering component of $g$ and lies in the periodic component of $g\circ f$ and periodic component of $f\circ g.$
\item[(ii)] $B_2$ lies in the preperiodic component of $f$ and wandering component of $g$ and lies in the preperiodic component of $g\circ f$ and periodic component of $f\circ g.$
\item[(iii)] $B_k,\;k\geq 3$ lies in the preperiodic component of $f$ and wandering component of $g$ and lies in the periodic component of $g\circ f$ and periodic component of $f\circ g.$
\item[(iv)] $G_k,\;k\geq 2$ lies in the wandering component of $f$ and wandering component of $g$ and also lies in the wandering component of $g\circ f$ and the wandering component of $f\circ g.$
\end{enumerate}
\end{remark}
\begin{theorem}\label{sec2,thm11}
There exist transcendental entire functions $f$ and $g$ having  a domain which lies in the periodic component of $f$ and wandering component of $g$ and lies in the  preperiodic component of $g\circ f$ and wandering component of $f\circ g.$
\end{theorem}
\begin{proof}
Let $S$ be the Carleman set defined as in Theorem \ref{sec2,thm1}.
Using the continuity of $e^z,$ for each $k=1,2,\ldots,$ choose $\eta_k,\xi_k$  such  that
\[|e^w+(4k+6)|<\frac{1}{2},\; \mbox{whenever}\; |w-(\pi i+\log(4k+6))|<\eta_k,\] and
\[|e^w-(4k+6)|<\frac{1}{2},\; \mbox{whenever}\; |w-\log(4k+6)|<\xi_k.\]
 Also choose $\delta_0,\delta_1,\delta_2,\delta_3,\delta_4,\delta_5$ so that
\[|e^w-2|<\frac{1}{2},\; \mbox{whenever}\; |w-\log 2|<\delta_0,\]
\[|e^w+6|<\frac{1}{2},\; \mbox{whenever}\; |w-(\pi i+\log 6)|<\delta_1,\]
\[|e^w-6|<\frac{1}{2},\; \mbox{whenever}\; |w-\log 6|<\delta_2,\]
\[|e^w+10|<\frac{1}{2},\; \mbox{whenever}\; |w-(\pi i+\log 10)|<\delta_3,\]
\[|e^w+14|<\frac{1}{2},\; \mbox{whenever}\; |w-(\pi i+\log 14)|<\delta_4,\] and
\[|e^w+18|<\frac{1}{2},\; \mbox{whenever}\; |w-(\pi i+\log 18)|<\delta_5.\]
Define
\begin{equation}\notag
\alpha(z)=
\begin{cases}
\log 2,  & z\in G_0\cup\B{\{{\bigcup}_{k=1}^{\ity}( L_k\cup M_k)\B\}}\\
 \pi i+\log 6,   & z\in G_1\\
\log 6,   & z\in B_1\\ 
 \pi i+\log 14,   & z\in B_2\\
 \pi i+\log 18,   & z\in B_3\\
 \pi i+\log 6,   & z\in  B_4\\
 \pi i+\log(4k+6),   & z\in B_k,\;k=5,6,\ldots\\
\log(4k+6),   & z\in G_k,\;k=2,3,\ldots
\end{cases}
\end{equation}
\begin{equation}\notag
\beta(z)=
\begin{cases}
\log 2,  & z\in G_0\cup\B{\{{\bigcup}_{k=1}^{\ity}( L_k\cup M_k)\B\}}\\
 \pi i+\log 10,   & z\in B_1\\
\log 6,   & z\in B_2\\ 
 \pi i+\log 18,   & z\in B_3\\
\pi i+\log 6,   & z\in  B_4\\
 \pi i+\log(4k+6),   & z\in B_k,\;k=5,6,\ldots\\
\log(4k+6),   & z\in G_k,\;k=1,2,\ldots
\end{cases}
\end{equation}
\begin{equation}\notag
\epsilon(z)=
\begin{cases}
\delta_0,   & z\in G_0\cup\B{\{{\bigcup}_{k=1}^{\ity}( L_k\cup M_k)\B\}}\\
\delta_1,   & z\in G_1\\
\delta_2,   & z\in B_1\\
\delta_4,   & z\in B_2\\
\delta_5,   & z\in B_3\\
\delta_1,   & z\in B_4\\
\eta_k,  & z\in B_k,\;k=5,6,\ldots\\
\xi_k,   & z\in G_k,\;k=2,3,\ldots\\
\end{cases}
\end{equation}
and
\begin{equation}\notag
\epsilon_1(z)=
\begin{cases}
\delta_0,   & z\in G_0\cup\B{\{{\bigcup}_{k=1}^{\ity}( L_k\cup M_k)\B\}}\\
\delta_3,   & z\in B_1\\
\delta_2,   & z\in B_2\\
\delta_5,   & z\in B_3\\
\delta_1,   & z\in B_4\\
\xi_k,   & z\in G_k,\;k=1,2,\ldots\\
\eta_k,   & z\in B_k,\;k=5,6,\ldots\\
\end{cases}
\end{equation}
Clearly  $\alpha(z)$ is continuous on $S$ and analytic in $S^\circ.$ Thus there is an entire function $\gamma(z)$ such that $|\gamma(z)-\alpha(z)|<\epsilon(z)$ for all $z\in S.$ The function $f(z)=e^{\gamma(z)}$ is an entire function and it satisfies
\[|f(z)-2|<\frac{1}{2},\quad z\in G_0\cup\B{\{{\bigcup}_{k=1}^{\ity}( L_k\cup M_k)\B\}},\]
\[|f(z)+6|<\frac{1}{2},\quad z\in G_1,\]
\[|f(z)-6|<\frac{1}{2},\quad z\in B_1,\]
\[|f(z)+14|<\frac{1}{2},\quad z\in B_2,\]
\[|f(z)+18|<\frac{1}{2},\quad z\in B_3,\]
\[|f(z)+6|<\frac{1}{2},\quad z\in B_4,\]
\[|f(z)+(4k+6)|<\frac{1}{2},\quad z\in B_k,\;k=5,6,\ldots,\] and
\[|f(z)-(4k+6)|<\frac{1}{2},\quad z\in G_k,\;k=2,3,\ldots.\]
Also $\beta(z)$ is a continuous function on $S$ and analytic in $S^\circ,$  there exist an entire function $\gamma_1(z)$  such that if $g(z)=e^{\gamma_1(z)}$ then
\[|g(z)-2|<\frac{1}{2},\quad z\in G_0\cup\B{\{{\bigcup}_{k=1}^{\ity}( L_k\cup M_k)\B\}},\]
\[|g(z)+10|<\frac{1}{2},\quad z\in B_1,\]
\[|g(z)-6|<\frac{1}{2},\quad z\in B_2,\]
\[|g(z)+18|<\frac{1}{2},\quad z\in B_3,\]
\[|g(z)+6|<\frac{1}{2},\quad z\in B_4,\]
\[|g(z)+(4k+6)|<\frac{1}{2},\quad z\in B_k,\;k=5,6,\ldots,\] and
\[|g(z)-(4k+6)|<\frac{1}{2},\quad z\in G_k,\;k=1,2,\ldots.\]
It can now be easily seen that $G_1$ lies in the periodic component of $f$ and wandering component of $g$ and  lies in the preperiodic component of $g\circ f$ and wandering component of $f\circ g.$
\end{proof}
\begin{remark}\label{sec2,rem4}
\begin{enumerate}
\item[(i)] $G_k,\;k\geq 2$ lies in the wandering component of $f$ and wandering component of $g$ and also lies in the wandering component of $g\circ f$ and the wandering component of $f\circ g.$
\item[(ii)] $B_1$ lies in the periodic component of $f$ and wandering component of $g$ and lies in the wandering component of $g\circ f$ and periodic component of $f\circ g.$
\item[(iii)] $B_2$ lies in the preperiodic component of $f$ and wandering component of $g$ and lies in the periodic component of $g\circ f$ and preperiodic component of $f\circ g.$
\item[(iv)] $B_3$ lies in the preperiodic component of $f$ and wandering component of $g$ and lies in the wandering component of $g\circ f$ and periodic component of $f\circ g.$
\item[(v)] $B_4$ lies in the preperiodic component of $f$ and wandering component of $g$ and lies in the periodic component of $g\circ f$ and wandering component of $f\circ g.$
\item[(vi)] $B_k,\;k\geq 5$ lies in the wandering component of $f$ and wandering component of $g$ and also lies in the wandering component of $g\circ f$ and the wandering component of $f\circ g.$
\end{enumerate}
\end{remark}
\begin{theorem}\label{sec2,thm13}
There exist transcendental entire functions $f$ and $g$ having  infinite number of  domains which lies in the preperiodic component of $f$ and  preperiodic component of $g$ and also lies in the  preperiodic component of $g\circ f$ and preperiodic component of $f\circ g.$
\end{theorem}
\begin{proof}
Let $S$ be the Carleman set defined as in Theorem \ref{sec2,thm1}.
Using the continuity of $e^z,$ for each $k=1,2,\ldots,$ choose  $\delta_0,\delta_1,\delta_2,\mu_k,\nu_k$ such that
\[|e^w-2|<\frac{1}{2},\; \mbox{whenever}\; |w-\log 2|<\delta_0,\]
\[|e^w+6|<\frac{1}{2},\; \mbox{whenever}\; |w-(\pi i+\log 6)|<\delta_1,\]
\[|e^w+10|<\frac{1}{2},\; \mbox{whenever}\; |w-(\pi i+\log 10)|<\delta_2,\]
\[|e^w+(4k-2)|<\frac{1}{2},\; \mbox{whenever}\; |w-(\pi i+\log(4k-2))|<\mu_k,\] and
\[|e^w-(4k-2)|<\frac{1}{2},\; \mbox{whenever}\; |w-\log(4k-2)|<\nu_k.\]
We next define
\begin{equation}\notag
\alpha(z)=
\begin{cases}
\log 2,  & z\in G_0\cup\B{\{{\bigcup}_{k=1}^{\ity}( L_k\cup M_k)\B\}}\\
 \pi i+\log 6,   & z\in G_k,\;k=1,2,\ldots\\
 \pi i+\log 10,   & z\in B_1\\
\pi i+\log(4k-2),   & z\in B_k,\;k=2,3,\ldots\\
\end{cases}
\end{equation}
\begin{equation}\notag
\beta(z)=
\begin{cases}
\log 2,  & z\in G_0\cup\B{\{{\bigcup}_{k=1}^{\ity}( L_k\cup M_k)\B\}}\\
\pi i+\log 6,   & z\in  G_1\\
 \pi i+\log 10,   & z\in B_1\\
 \pi i+\log(4k-2),   & z\in B_k,\;k=2,3,\ldots\\
\log(4k-2),   & z\in G_k,\;k=2,3,\ldots
\end{cases}
\end{equation}
\begin{equation}\notag
\epsilon(z)=
\begin{cases}
\delta_0,   & z\in G_0\cup\B{\{{\bigcup}_{k=1}^{\ity}( L_k\cup M_k)\B\}}\\
\delta_1,   & z\in G_k,\;k=1,2,\ldots\\
\delta_2,   & z\in B_1\\
\mu_k,  & z\in B_k,\;k=2,3,\ldots\\
\end{cases}
\end{equation}
and
\begin{equation}\notag
\epsilon_1(z)=
\begin{cases}
\delta_0,   & z\in G_0\cup\B{\{{\bigcup}_{k=1}^{\ity}( L_k\cup M_k)\B\}}\\
\delta_1,   & z\in G_1\\
\delta_2,   & z\in B_1\\
\mu_k,   & z\in B_k,\;k=2,3,\ldots\\
\nu_k,   & z\in G_k,\;k=2,3,\ldots\\
\end{cases}
\end{equation}
Clearly  $\alpha(z)$ is continuous on $S$ and analytic in $S^\circ.$ So there is an entire function $\gamma(z)$ such that $|\gamma(z)-\alpha(z)|<\epsilon(z)$ for all $z\in S.$ The function $f(z)=e^{\gamma(z)}$ is an entire function and it satisfies
\[|f(z)-2|<\frac{1}{2},\quad z\in G_0\cup\B{\{{\bigcup}_{k=1}^{\ity}( L_k\cup M_k)\B\}},\] 
\[|f(z)+6|<\frac{1}{2},\quad z\in G_k,\;k=1,2,\ldots,\]
\[|f(z)+10|<\frac{1}{2},\quad z\in B_1,\] and
\[|f(z)+(4k-2)|<\frac{1}{2},\quad z\in B_k,\;k=2,3,\ldots.\]
Also $\beta(z)$ is  continuous  on $S$ and analytic in $S^\circ,$  there exist an entire function $\gamma_1(z)$  such that if $g(z)=e^{\gamma_1(z)}$ then
\[|g(z)-2|<\frac{1}{2},\quad z\in G_0\cup\B{\{{\bigcup}_{k=1}^{\ity}( L_k\cup M_k)\B\}},\]
\[|g(z)+6|<\frac{1}{2},\quad z\in G_1,\]
\[|g(z)+10|<\frac{1}{2},\quad z\in B_1,\]
\[|g(z)+(4k-2)|<\frac{1}{2},\quad z\in B_k,\;k=2,3,\ldots,\] and
\[|g(z)-(4k-2)|<\frac{1}{2},\quad z\in G_k,\;k=2,3,\ldots.\]
It can now be easily seen that each $G_k,\;k=1,2,\ldots$ lies in the preperiodic component of $f$  and  preperiodic component of $g$ and also lies in the  preperiodic component of $g\circ f$ and preperiodic component of $f\circ g.$
\end{proof}
\begin{remark}\label{sec2,rem6}
\begin{enumerate}
\item[(i)] Both $B_1$ and $B_2$ lies in the periodic component of $f$ and periodic component of $g$ and also lies in the periodic component of $g\circ f$ and periodic component of $f\circ g.$
\item[(ii)] $B_k,\;k\geq 3$ lies in the preperiodic component of $f$ and preperiodic component of $g$ and also lies in the preperiodic component of $g\circ f$ and preperiodic component of $f\circ g.$
\end{enumerate}
\end{remark}
\begin{theorem}\label{sec2,thm15}
There exist transcendental entire functions $f$ and $g$ having a domain which lies in the preperiodic component of $f$ and  preperiodic component of $g$ and also lies in the  preperiodic component of $g\circ f$ but lies in the wandering component of $f\circ g.$
\end{theorem}
\begin{proof}
Let $S$ be the Carleman set defined as in Theorem \ref{sec2,thm1}.
Using the continuity of $e^z,$ for each $k=1,2,\ldots,$ choose $\eta_k,\xi_k$  so that
\[|e^w+(4k+6)|<\frac{1}{2},\; \mbox{whenever}\; |w-(\pi i+\log(4k+6))|<\eta_k,\] and
\[|e^w-(4k+6)|<\frac{1}{2},\; \mbox{whenever}\; |w-\log(4k+6)|<\xi_k.\]
 Also choose $\delta_0,\delta_1,\delta_2,\delta_3,\delta_4$ so that
\[|e^w-2|<\frac{1}{2},\; \mbox{whenever}\; |w-\log 2|<\delta_0,\]
\[|e^w+6|<\frac{1}{2},\; \mbox{whenever}\; |w-(\pi i+\log 6)|<\delta_1,\]
\[|e^w-10|<\frac{1}{2},\; \mbox{whenever}\; |w-\log 10|<\delta_2,\]
\[|e^w+10|<\frac{1}{2},\;\mbox{whenever}\; |w-(\pi i+\log 10)|<\delta_3,\] and
\[|e^w+14|<\frac{1}{2},\; \mbox{whenever}\; |w-(\pi i+\log 14)|<\delta_4.\]
We next define
\begin{equation}\notag
\alpha(z)=
\begin{cases}
\log 2,  & z\in G_0\cup\B{\{{\bigcup}_{k=1}^{\ity}( L_k\cup M_k)\B\}}\\
 \pi i+\log 6,   & z\in G_1\\
\log 10,   & z\in B_1\\ 
 \pi i+\log 6,   & z\in G_2\\
 \pi i+\log(4k+6),   & z\in B_k,\;k=2,3,\ldots\\
\log(4k+6),   & z\in G_k,\;k=3,4,\ldots
\end{cases}
\end{equation}
\begin{equation}\notag
\beta(z)=
\begin{cases}
\log 2,  & z\in G_0\cup\B{\{{\bigcup}_{k=1}^{\ity}( L_k\cup M_k)\B\}}\\
 \pi i+\log 14,   & z\in G_1\\
\log 10,   & z\in  B_1\\
 \pi i+\log 14,   & z\in B_2\\
\pi i+\log 10,   & z\in  B_3\\
\log(4k+6),   & z\in G_k,\;k=2,3,\ldots\\
 \pi i+\log(4k+6),   & z\in B_k,\;k=4,5,\ldots\\
\end{cases}
\end{equation}
\begin{equation}\notag
\epsilon(z)=
\begin{cases}
\delta_0,   & z\in G_0\cup\B{\{{\bigcup}_{k=1}^{\ity}( L_k\cup M_k)\B\}}\\
\delta_1,   & z\in G_1\\
\delta_2,   & z\in B_1\\
\delta_1,   & z\in G_2\\
\eta_k,  & z\in B_k,\;k=2,3,\ldots\\
\xi_k,   & z\in G_k,\;k=3,4,\ldots\\
\end{cases}
\end{equation}
and
\begin{equation}\notag
\epsilon_1(z)=
\begin{cases}
\delta_0,   & z\in G_0\cup\B{\{{\bigcup}_{k=1}^{\ity}( L_k\cup M_k)\B\}}\\
\delta_4,   & z\in G_1\\
\delta_2,   & z\in B_1\\
\delta_4,   & z\in B_2\\
\delta_3,   & z\in B_3\\
\xi_k,   & z\in G_k,\;k=2,3,\ldots\\
\eta_k,   & z\in B_k,\;k=4,5,\ldots\\
\end{cases}
\end{equation}
Clearly  $\alpha(z)$ is continuous on $S$ and analytic in $S^\circ.$ Thus there is an entire function $\gamma(z)$ such that $|\gamma(z)-\alpha(z)|<\epsilon(z)$ for all $z\in S.$ The function $f(z)=e^{\gamma(z)}$ is an entire function and it satisfies
\[|f(z)-2|<\frac{1}{2},\quad z\in G_0\cup\B{\{{\bigcup}_{k=1}^{\ity}( L_k\cup M_k)\B\}},\]
\[|f(z)+6|<\frac{1}{2},\quad z\in G_1,\]
\[|f(z)-10|<\frac{1}{2},\quad z\in B_1,\]
\[|f(z)+6|<\frac{1}{2},\quad z\in G_2,\]
\[|f(z)+(4k+6)|<\frac{1}{2},\quad z\in B_k,\;k=2,3,\ldots,\] and
\[|f(z)-(4k+6)|<\frac{1}{2},\quad z\in G_k,\;k=3,4,\ldots.\]
Also $\beta(z)$ is a continuous function on $S$ and analytic in $S^\circ,$  there exist an entire function $\gamma_1(z)$  such that if $g(z)=e^{\gamma_1(z)}$ then
\[|g(z)-2|<\frac{1}{2},\quad z\in G_0\cup\B{\{{\bigcup}_{k=1}^{\ity}( L_k\cup M_k)\B\}},\]
\[|g(z)+14|<\frac{1}{2},\quad z\in G_1,\]
\[|g(z)-10|<\frac{1}{2},\quad z\in B_1,\]
\[|g(z)+14|<\frac{1}{2},\quad z\in B_2,\]
\[|g(z)+10|<\frac{1}{2},\quad z\in B_3,\]
\[|g(z)-(4k+6)|<\frac{1}{2},\quad z\in G_k,\;k=2,3,\ldots,\] and
\[|g(z)+(4k+6)|<\frac{1}{2},\quad z\in B_k,\;k=4,5,\ldots.\]
It can now be easily seen that $G_1$ lies in the preperiodic component of $f$ and preperiodic component of $g$ and also lies in the preperiodic component of $g\circ f$ but lies in the wandering component of $f\circ g.$
\end{proof}
\begin{remark}\label{sec2,rem7}
\begin{enumerate}
\item[(i)] $B_1$ lies in the periodic component of $f$ and wandering component of $g$ and lies in the wandering component of $g\circ f$ and periodic component of $f\circ g.$
\item[(ii)] $G_2$ lies in the periodic component of $f$ and wandering component of $g$ and lies in the periodic component of $g\circ f$ and wandering component of $f\circ g.$
\item[(iii)] $G_k,\;k\geq 3$ lies in the wandering component of $f$ and wandering component of $g$ and also lies in the wandering component of $g\circ f$ and wandering component of $f\circ g.$
\item[(iv)] $B_2$ lies in the wandering component of $f$ and periodic component of $g$ and lies in the periodic component of $g\circ f$ and wandering component of $f\circ g.$
\item[(v)] $B_3$ lies in the wandering component of $f$ and periodic component of $g$ and lies in the wandering component of $g\circ f$ and periodic component of $f\circ g.$
\item[(vi)] $B_k,\;k\geq 4$ lies in the wandering component of $f$ and wandering component of $g$ and also lies in the wandering component of $g\circ f$ and wandering component of $f\circ g.$
\end{enumerate}
\end{remark}
Acknowledgement. We are thankful to Prof. A. P. Singh, Central University of Rajasthan for his helpful comments and suggestions.

\end{document}